\documentclass[a4paper,12pt]{article}

\usepackage{amsmath,amssymb,amsthm}
\usepackage{mathrsfs}
\usepackage{eucal}
\usepackage{eufrak}
\usepackage[matrix,arrow,curve]{xy}
\usepackage[left=2cm,right=2cm,top=2cm,bottom=2cm,bindingoffset=0cm]{geometry}

\newtheorem*{thma}{Theorem}
\newtheorem{rem}{Remark}[section]
\newtheorem{lemma}{Lemma}[section]
\newtheorem{cor}{Corollary}[section]
\newtheorem{prop}{Proposition}[section]

\numberwithin{equation}{section}

\def\eps{\varepsilon}
\def\del{\delta}

\def\A{\EuScript {A}}

\def\bar{\begin{array}}
\def\ear{\end{array}}

\def\beq{\begin{equation}}
\def\eeq{\end{equation}}

\def\Thn{\Theta_{\mathrm{null}}}

\def\M{\CMcal{M}_g}

\def\Mt{\CMcal{T}_g}

\def\cM{\overline{\CMcal{M}}_g}

\def\Se{\CMcal{S}^+_g}
\def\cSe{\overline{\CMcal{S}}^+_g}
\def\Ste{\tilde{\CMcal{S}}^+_g}

\def\So{\CMcal{S}^-_g}
\def\cSo{\overline{\CMcal{S}}^-_g}

\def\Pic{\mathrm{Pic}}

\def\dim{\mathrm{dim\,}}
\def\div{\mathrm{div}}

\def\Ad{\CMcal{H}_g}

\def\Ado{\CMcal{H}^-_g([2^{g-1}])}

\def\tAdo{\tilde{\CMcal{H}}^-_g([2^{g-1}])}

\def\B{\CMcal{B}}

\def\A{\CMcal{A}}

\def\Z{\CMcal{Z}_g}

\def\s{\varsigma}

\def\M{\CMcal{M}_g}

\def\cM{\overline{\CMcal{M}}_g}

\def\U{\CMcal{U}}
\def\V{\CMcal{V}}

\def\a{\rho}
\def\t{\vartheta}

\def\L{\CMcal{L}}

\begin{document}

\author{Mikhail Basok \thanks{This work was supported by the Chebyshev Laboratory  (Department of Mathematics and Mechanics, St. Petersburg State University) under RF Government grant 11.G34.31.0026, by JSC "Gazprom Neft" and by RFBR grant 12-01-31492.}}
\title{\LARGE\bf Tau function and moduli of spin curves.}
\date{}
\maketitle

\bigskip

\begin{abstract}
The goal of the paper is to give an analytic proof of the formula of G. Farkas for the divisor class of spinors with multiple zeros in the moduli space of odd spin curves. We make use of the technique developed by Korotkin and Zograf that is based on properties of the Bergman tau function. We also show how the Farkas formula for the {\it theta-null} in the rational Picard group of the moduli space of even spin curves can be derived from classical theory of theta functions.
\end{abstract}

\section{The moduli space of odd spin curves.}

Let $\M$ be the moduli space of smooth genus $g$ algebraic curves, assume that $g\geq 3$. Let $\cM$ be its Deligne-Mumford compactification. The boundary $\cM\smallsetminus\M$ consists of $\left [ \frac g2 \right ] + 1$ irreducible divisors $\Delta_0,\dots, \Delta_{\left [ \frac g2 \right ] }$ where $\Delta_0$ is the closure of the locus of irreducible curves with one node and $\Delta_j$ for $j\geq 1$ is the closure of the locus of reducible one-nodal curves.

The moduli space $\So$ of {\it smooth} odd spin curves is $2^{g-1}(2^g-1)$ cover of $\M$. The cover is extended to a branched cover of $\cM$ by the Cornalba compactification $\cSo$ of $\So$ ramified over~$\Delta_0$.

{\bf Cornalba compactification.} A nodal curve $C$ is called {\it quasi-stable} if it satisfies two conditions:

1) Every rational component $E$ of $C$ intersects $\overline{C \smallsetminus E}$ at two or more points;

2) Any two rational components $E_1,E_2$ of $C$ such that $\#\, E_i\cap \overline{C \smallsetminus E_i} = 2$ are disjoint.

\noindent Rational component $E$ of $C$ intersecting $\overline{C \smallsetminus E}$ at exactly two points is called {\it exceptional}.

Following~\cite{COR} we define a {\it spin curve} as a triple $(C,\eta,\beta)$ consisting of a quasi-stable curve $C$, a line bundle $\eta$ of degree $g-1$ on it and a homomorphism $\beta:\eta^{\otimes 2}\to \omega_C$ with the following properties:

1) $\eta$ is of degree one on every exceptional component of $C$;

2) $\beta$ does not identically vanish on every non-exceptional component of $C$.

\noindent The parity of the spin curve $(C,\eta,\beta)$ is the parity of $\dim H^0(C,\eta)$. The parity is invariant under continuous deformations (see \cite{MUM} or \cite{ATJ}).

An isomorphism between $(C,\eta,\beta)$ and $(C',\eta',\beta')$ is an isomorphism $\sigma~:~C\to~C'$ such that $\sigma^*\eta'$ and $\eta$ are isomorphic and the following diagram
$$
\xymatrix{
\eta^2 \ar[rr]^(0.45){\rho\otimes\rho} \ar[d]^{\beta}&& \sigma^*(\eta' )^2\ar[d]^{\sigma^*\beta'}\\
\omega_C\ar[rr]^(0.45){\simeq} && \sigma^* \omega_{C'}
}
$$
is commutative, where $\rho$ is an isomorphism between $\eta$ and $\sigma^*\eta'$. The moduli space $\cSo$ consists of all equivalence classes of odd spin curves under such isomorphisms. The projection $\chi~:~\cSo~\to~\cM$ maps (an equivalence class of) a triple $(C,\eta,\beta)$ to (an equivalence class of) a curve $\tilde{C}$ which is obtained from $C$ by contracting all exceptional components to points.

{\bf Rational Picard group of $\cSo$.} The boundary $\cSo\smallsetminus\So$ is the union of irreducible divisors $A_0,\dots,A_{\left[ g/2 \right]}, B_0,\dots,B_{\left[ g/2 \right]}$ such that $\chi (A_j) = \chi(B_j) = \Delta_j$ for $j = 0,\dots,\left [ \frac g2 \right ]$.

{\it Description of $A_j$ and $B_j$ for $j \neq 0$.} Note that there are no spin curves $(C,\eta,\beta)$ with a reducible one-nodal base curve $C$, since the relative dualizing sheaf $\omega_C$ on a reducible curve with one node being restricted to each component must be of odd degree (see~\cite{COR} for more details).

Let $(C,\eta,\beta)$ be a spin curve such that $C = C_1\cup E\cup C_2$ where $C_1$ and $C_2$ are smooth curves of genus $j$ and $g-j$ respectively and $E$ is an exceptional component. The divisor $A_j$ parametrizes the closure of the locus of such curves with the property that $\eta$ restricted to $C_1$ is odd. The divisor $B_j$ is the closure of the locus of the same type spin curves such that $\eta$ restricted to $C_1$ is even. Formally
$$
\bar{l}
A_j = Cl\{(C_1\cup E\cup C_2, \eta, \beta)\in \cSo\ :\ \eta|_{C_1} \text{ is odd} \},\\
B_j = Cl\{(C_1\cup E\cup C_2, \eta, \beta)\in \cSo\ :\ \eta|_{C_1} \text{ is even} \},
\ear
$$
where $Cl$ stands for the closure.

{\it Description of $A_0$ and $B_0$.} Unlike the case $j\neq 0$ a spin curve $(C,\eta,\beta)$ such that $\chi(C,\eta,\beta)$ is an irreducible one-nodal curve, does not necessary have exceptional components. Let $A_0$ parametrize the closure of the locus of spin curves with one-nodal irreducible underlying curve and $B_0$ parametrize the closure of the locus of spin curves mapping to $\Delta_0$ under $\chi$ and having an exceptional component. Formally
$$
\bar{l}
A_0 = Cl\{  (C/_{p\sim q}, \eta, \beta)\in \cSo\ :\ \text{$C$ is smooth curve of genus $g-1$},\ p,q\in C   \},\\
B_0 = Cl\{  (C\cup E, \eta, \beta)\in \cSo\ :\ \text{$C$ is smooth curve of genus $g-1$, $E$ is exceptional}  \}.\\
\ear
$$

Denote by $\alpha_j$ and $\beta_j$ the classes of $A_j$ and $B_j$ in the rational Picard group $\Pic(\cSo)~\otimes~\mathbb Q$ respectively. Let $\lambda$ be the pullback of the Hodge class on $\cM$ under $\chi$. The Picard group is generated by the classes
\beq\label{picgen}
\Pic(\cSo)\otimes \mathbb Q = span_{\mathbb Q}(\lambda, \alpha_0,\dots,\alpha_{\left [ \frac g2 \right ]}, \beta_0,\dots,\beta_{\left [ \frac g2 \right ]}).
\eeq

Consider the following divisor on $\cSo$:
$$
\Z = Cl\{ (C,\eta)\in \So\ |\ \eta = \CMcal{O}_C (2x_1+x_2+\dots+x_{g-2})   \}.
$$
The class of $\Z$ in the rational Picard group $\Pic(\cSo)~\otimes~\mathbb Q$ can be expressed as a linear combination of generators~\eqref{picgen}. G. Farkas determined the coefficients in this expansion and used it for the birational classification of moduli spaces of odd spin curves (see~\cite{FARo}). The goal of this paper is to show how this coefficients can be computed analytically from properties of the Bergman tau function on the moduli space of abelian differentials.

We also study the theta-null divisor on the moduli space of even spin curves. We show how to express the theta-null in terms of standard generators of the rational Picard group in the framework of the classical theory of theta functions. This expression was also obtained by G. Farkas in his work~\cite{FARe} by different methods. G. Farkas used this expression for the birational classification of the moduli space of even spin curves.

The paper is organized as follows: we introduce the Bergman tau function and list its basic properties in Section~\ref{tausection}. In Section~\ref{thetasection} we study the asymptotics of the theta function under a degeneration of a curve; this asymptotics is well-known (see~\cite{FAY}) but we write it down to fix notations. Then in Section~\ref{farforsec} we construct an odd spinor using the theta function and analyze the behavior of the tau function on the space of squares of these odd spinors. This results in the Farkas formula for $\Z$. Finally in Section~\ref{thetanullsection} the paper we derive the formula for theta-null.

\section{The Bergman tau function on moduli spaces of holomorphic differentials with double zeros.}\label{tausection}

Let $\Ad$ denote the moduli space of holomorphic differentials on smooth genus $g$ curves (see~\cite{KonZor}). This space admits a natural stratification according to multiplicities of zeros of the differential. Denote by $\Ad([2^{g-1}])$ the stratum corresponding to differentials with $g-1$ distinct zeros of multiplicity two. Let $C$ be a genus $g$ curve and $\omega$ be a differential on $C$ such that  $(C,\omega)\in \Ad([2^{g-1}])$. If $\div\omega = 2D$ then the linear system $|D|$ corresponds to a spin bundle on $L\to C$. Let $\Ado$ be the connected component of $\Ad([2^{g-1}])$ corresponding to the case when $L$ is an odd spin bundle (see~\cite{ConComp}).

{\bf Homological coordinates.}
Denote by $\Mt$ the moduli space of Torelli marked curves (i. e. curves with a fixed symplectic basis in $H_1(C)$), and let $\tAdo$ be the cover of $\Ado$ induced by the forgetful map $\Mt\to \M$.

Fix an arbitrary point $(C,\a, \omega )\in \tAdo$, where we denote the Torelli marking by $\a$. Let $p_1,\dots,p_{g-1}\in C$ be the zeros of $\omega$. Consider simple  non-intersecting paths $l_j$ connecting $p_{g-1}$ with $p_j$ for $j = 1,\dots, g-2$. Let $a_1,\dots,a_g,b_1,\dots,b_g$ be simple loops on $C\smallsetminus\{p_1,\dots,p_{g-1}\}$ representing that do not intersect $\{l_j\}_{j = 1}^{g-2}$. Denote by $s_1,\dots,s_{3g-2}$ the basis in $H_1(C~\smallsetminus~\{p_1,\dots,p_{g-1}\})$ dual to the basis represented by $a_1,\dots,a_g,b_1,\dots,b_g,l_1,\dots,l_{g-2}$ in the relative homology group $H_1(C, \{p_1,\dots,p_{g-1}\})$; we have $s_j = -b_j,s_{g+j} = a_j$ and $s_{2g+j}$ is homologous to a small positive oriented circle around $p_j$. Define a set of local coordinates near $(C,\{a_j,b_j\}_{j = 1}^g, \omega )$ (see~\cite{KonZor}):
$$
\bar{l}
z_j = \int_{a^{\circ}_j} \omega,\quad j = 1,\dots,g,\\
z_{j+g} = \int_{b^{\circ}_j} \omega,\quad j = 1,\dots,g,\\
z_{j+2g} = \int_{l_j} \omega,\quad j = 1,\dots,g-2.
\ear
$$
Following Kontsevich and Zorich we call these coordinates {\it homological}.

{\bf Definition of the tau function and its basic properties.} For any differential $\omega$ on $C$ introduce the meromorphic projective connection $S_{\omega} = \frac{\omega''}{\omega} - \frac 32 \left ( \frac{\omega'}{\omega}\right )^2$ (that is, the Schwarzian differential of the abelian integral $\int^x \omega $ with respect to a local parameter $\zeta$ on $C$). Let $\B(x,y)$ be the canonical bidifferential on $C$ (that is the symmetric bidifferential with a quadratic pole on the diagonal with biresidue 1 and normalized for zero $a$-periods). In terms of a local parameter $\zeta$ on $C$ one has the following expansion:
$$
\B(x,y) = \left ( \frac{1}{(\zeta(x) - \zeta(y))^2} + \frac{S_B(\zeta(x))}{6} + O(\zeta(x) - \zeta(y))^2 \right )\,d\zeta(x)d\zeta(y)\quad \text{as }x\to y.
$$
Clearly $S_B$ is a projective connection; this projective connection is called the {\it Bergman projective connection}. The difference of the two projective connections $S_B - S_{\omega}$ is a meromorphic quadratic differential on $C$.  Introduce a connection on the trivial line bundle on $\tAdo$ by the formula
$$
d_{B} = d +\frac{6}{\pi i} \sum\limits_{j = 1}^{3g - 2} \left ( \int_{s_j} \frac{S_B - S_{\omega}}{\omega}\right )dz_j.
$$
As it was shown in~\cite{KK} this connection is flat. The {\it tau function} $\tau = \tau (C,\{a_j,b_j\}_{j = 1}^g,\omega)$ is defined up to a constant factor\footnote{In fact this is the $72$-th power of the Bergman tau function studied in \cite{KK}. The name "tau function" is due to the relation of $\tau$ to isomonodromic Jimbo-Miwa tau function in the case of Hurwitz spaces.} as a horizontal (covariant constant) section of the trivial line bundle on $\tAdo$. In other words, $\tau:\tAdo\to \mathbb C$ is a holomorphic function such that
\beq\label{taudefin}
d_B\, \tau = 0.
\eeq
A solution of~\eqref{taudefin} was explicitly constructed in~\cite{KK}.

The group $Sp(2g, \mathbb Z) \times \mathbb C^*$ acts naturally on $\tAdo$ by changing the Torelli marking and multiplying the differential by a nonzero complex number. Note that $\tAdo /{Sp(2g,\mathbb Z)}$ coincides with $\Ado$.

Consider a natural map $\pi: \Ado/\mathbb C^* \to \So$ which assigns to a differential the spin bundle associated with the square root of the differential. The map $\pi$ is generally one-to-one, since an odd spin bundle generically has one-dimensional space of holomorphic sections. The image of $\pi$ is~$\So\smallsetminus \Z$.

\begin{lemma}[see~\cite{KK} for the proof]
\label{tauprop}
The tau function has the following properties:

1) $\tau$ is a nowhere vanishing holomorphic function on $\tAdo$.

2) For any $t\in \mathbb C^*$
$$
\tau(C,\{a_j,b_j\}_{j = 1}^g, t\, \omega) = t^{16(g-1)}\, \tau(C,\{a_j,b_j\}_{j = 1}^g, \omega).
$$

3) For any symplectic transformation $\sigma$ in $H_1(C)$
$$
\tau(C,\{\sigma(a_j),\sigma(b_j)\}_{j = 1}^g, \omega) = \mathrm{det} (\sigma_{12}\Omega + \sigma_{11})^{72}\,\tau(C,\{a_j,b_j\}_{j = 1}^g, \omega),
$$
where $\sigma = \begin{pmatrix} \sigma_{11} & \sigma_{12} \\ \sigma_{21} & \sigma_{22} \end{pmatrix}$ in the basis $\{a_j,b_j\}_{j = 1}^g$.
\end{lemma}

Consider the tautological line bundle $\L~\to~\Ado/\mathbb C^*$ with respect to the action of $\mathbb C^*$. Let $\mathbb E_g$ be the pullback of the Hodge vector bundle on $\M$ to $\Ado/\mathbb C^*$. Denote by $\Lambda$ the corresponding determinant bundle $\bigwedge^g\mathbb E_g$.
\begin{cor}
\label{secdef}
The function $\tau$ can be naturally viewed as a nowhere vanishing holomorphic section of the line bundle $\mathrm{Hom} (\L^{16(g-1)}, \Lambda^{72})$ on $\Ado/\mathbb C^*$.
\end{cor}

\section{Asymptotic behavior of the theta function under a curve degeneration.}\label{thetasection}

All facts written down in this subsection are well-known and can be found in classical literature (see e. g.~\cite{FAY}).

We understand a {\it theta-characteristic} as a vector $\eta\in (\mathbb Z/2\mathbb Z)^{2g}$.  The parity of a theta characteristic $\eta$ is defined as
$$
\sum\limits_{j = 1}^g\eta_{2j}\eta_{2j-1}.
$$

{\bf The case of reducible curves.} Let $C = C_1\cup C_2$ be a one-nodal reducible curve of genus $g$ where $C_1, C_2$ are its smooth components. Denote the genus of $C_1$ by $j$. Let $\eta = \eta_1\oplus \eta_2$ be some theta characteristic such that $\eta_1\in (\mathbb Z/2\mathbb Z)^{2j}$ and $\eta_2\in (\mathbb Z/2\mathbb Z)^{2(g-j)}$.

Let $p_i\in C_i$, $i = 1,2$, be points such that $C = C_1\sqcup C_2/_{p_1\sim p_2}$ and let $\zeta_i:U_i\to \mathbb C$ be some local coordinate in a neighborhood of $p_i$ such that $\zeta_i(p_i) = 0$. For small $t\in \mathbb C$, consider a family of curves $C_t$ defined as follows:
\beq\label{degenred}
C_t = (C_1\smallsetminus U_1) \cup (C_2\smallsetminus U_2) \cup \{(x_1,x_2,t)\in U_1\times U_2\ |\ \zeta_1(x_1)\,\zeta_2(x_2) = t\}.
\eeq
Clearly $C_0$ is isomorphic to $C$ and $C_t$ is smooth when $t\neq 0$. Moreover, the image of a neighborhood of $0\in \mathbb C$ under the map $t\to C_t$ is transversal to $\Delta_j$.

Fix a Torelli marking $\a_1 = \{a_i,b_i\}_{i = 1}^j$ on $C_1$ and $\a_2 = \{a_i,b_i\}_{i = j+1}^g$ on $C_2$. The marking $\a_1\cup\a_2$ gives rise to a Torelli marking on $C_t$ for all $t\neq 0$.  Let $v_1,\dots,v_j\in H^0(C_1,\omega_{C_1})$ and $v_{j+1},\dots,v_g\in H^0(C_2,\omega_{C_2})$ be the bases of normalized holomorphic differentials on $C_1$ and $C_2$ respectively.

Denote the matrix of $b$-periods for $C_t$ with respect to the marking $\a_1\cup\a_2$ by $\Omega_t$. Let
$$
\theta[\eta](\cdot,\Omega_t):\mathbb C^g\to \mathbb C
$$
be the theta function corresponding to $\Omega_t$ with the characteristic $\eta$. Denote the matrices of $b$-periods on $C_1$ and $C_2$ by $\Omega_1$ and $\Omega_2$ respectively.

\begin{prop}\label{thasred}
Let $W_1 = (w_1,\dots,w_j)\in \mathbb C^j$ and $W_2 = (w_{j+1},\dots,w_g) \in \mathbb C^{g-j}$. Put $R_i = \frac{v_i}{d\zeta_1}|_{p_1}$ if $i\leq j$ and $R_i = \frac{v_i}{d\zeta_2}|_{p_2}$ if $i > j$. Then one has
$$
\bar{ll}
\theta[\eta](W,\Omega_t) = & \theta[\eta_1](W_1,\Omega_1)\,\theta[\eta_2](W_2,\Omega_2)\\
& + 4t\,\sum\limits_{i,k = 1}^g\frac{\partial^2}{\partial w_i\partial w_k}\Bigl(\theta[\eta_1](W_1,\Omega_1)\,\theta[\eta_2](W_2,\Omega_2)\Bigr) \,R_iR_k + O(t^2)
\ear$$
as $t\to 0$ uniformly on compact subsets of $\mathbb C^g$, where $W = W_1\oplus W_2\in \mathbb C^g$.
\end{prop}
\begin{proof}
Proposition immediately follows from the expansion (see~\cite[p. 41]{FAY})
$$
\Omega_t = \begin{pmatrix} \Omega_1 & 0 \\ 0 & \Omega_2 \end{pmatrix} + \frac{\pi i\,t}2\,\begin{pmatrix}R_1 \\ \vdots \\ R_g\end{pmatrix}\begin{pmatrix}R_1 & \ldots & R_g\end{pmatrix} + O(t^2).
$$
\end{proof}

{\bf The case of irreducible curves.} Let $C$ be a smooth curve of genus $g-1$ and let $p_1,p_2\in C$ be distinct points. Then $C/_{p_1\sim p_2}$ is irreducible one-nodal curve of genus $g$. Consider a theta characteristic $\eta = \eta_1 \oplus\begin{pmatrix}\eps \\ \del\end{pmatrix}$ such that $\eta_1\in (\mathbb Z/2\mathbb Z )^{2(g-1)}$ and $\eps,\del\in \mathbb Z/2\mathbb Z$.

Choose coordinates $\zeta_i:U_i\to \mathbb C$ near $p_i$ where $U_1\cap U_2 = \varnothing$ and $\zeta_i(p_i) = 0$. For any small $t\in \mathbb C$ consider a curve $C_t$ defined as follows:
\beq\label{degenirred}
C_t = (C\smallsetminus (U_1\cup U_2))\cup \{(x_1,x_2)\in U_1\times U_2\ |\ \zeta_1(x_1)\,\zeta_2(x_2) = t\}.
\eeq
Clearly $C_0 = C/_{p_1\sim p_2}$, whereas $C_t$ is smooth and have genus $g$ when $t\neq 0$. The image of $t\mapsto C_t$ is transversal to $\Delta_0$.

Let $a$ be a simple closed curve on $C$ about $p_1$ that lies outside of $U_1\cup U_2$. For a Torelli marking $\a$ on $C$ fix a set of loops $\{a_i,b_i\}_{i = 1}^{g-1}$ representing $\a$ and choose a simple path $b$ from $p_1$ to $p_2$ which does not intersect these loops. Then the set $\{a,b\}\cup \{a_i,b_i\}_{i = 1}^{g-1}$ gives rise to a Torelli marking on $C_t$ for small $t\in\mathbb C\smallsetminus\mathbb R_{\geq 0}$. Note that to achieve this  for all $t$ is impossible because of the monodromy $b\mapsto b+a$ when $t$ goes around zero.

Let $L_t = L_t(\eta)$ be the spin bundle on $C_t$ associated with $\eta$ and $\{a,b\}\cup \{a_i,b_i\}_{i = 1}^{g-1}$, and let $L_{t,0} = L_t(0)$ be the spin bundle with zero characteristic. By definition we have $(L_t\otimes L_{t,0}^*)(a) = \delta$ and $(L_t\otimes L_{t,0}^*)(b) = \eps$, where $(L_t\otimes L_{t,0}^*)\in H^1(C_t,\mathbb Z/2\mathbb Z)$ is the cohomology class of the zero-degree bundle $L_t\otimes L_{t,0}^*$.

Denote by $\Omega_t$ the matrix of $b$-periods of $C_t$ with respect to $\{a,b\}\cup \{a_i,b_i\}_{i = 1}^{g-1}$, and consider the corresponding theta function with the characteristic $\eta$:
$$
\theta[\eta](\cdot,\Omega_t):\mathbb C^g\to \mathbb C.
$$
Denote by $\Omega$ the matrix of $b$-periods on $C$ with respect to $\a$.

\begin{prop}\label{thasirred1}
Assume that $\delta = 1$. Then $\theta[\eta](\cdot,\Omega_t)$ is defined up to the $8$th root of unity and has the following asymptotics on every compact subset of $\mathbb C^g$
$$
\theta[\eta](w_1,\dots,w_g,\Omega_t) = t^{1/8} \Bigl (e^{-cw_g+r }\,\theta[\eta_1](w_1,\dots,w_{g-1},\Omega) +  e^{cw_g }\,\theta[\eta_1](w_1+c_1,\dots,w_{g-1}+c_{g-1},\Omega) + O(t) \Bigr),
$$
where $c,r,c_j$ are independent on $\{w_j\}$ but depend on moduli of curve and $c\neq 0$ and $\theta[\eta_1](c_1,\dots,c_{g-1},\Omega)\neq 0$ outside of some divisor in the moduli space $\overline{\CMcal{M}}_{g-1,2}$.
\end{prop}

\begin{rem}\label{anbound1}
Note that if $\delta = 1$ then the bundle $L_t\to C_t$ is well-defined for all small $t\in \mathbb C$, so one obtains a family $(C_t,L_t)$ of smooth spin curves. If $\eta$ is an odd characteristic then the closure of this family in $\cSo$ is transversal to the boundary divisor $A_0$.
\end{rem}

\begin{prop}\label{thasirred2}
Assume that $\delta = 0$. Then $\theta[\eta](\cdot,\Omega_t)$ depends on the choice of a branch of $\sqrt{t}$ and has the following asymptotics uniformly on compact subsets of $\mathbb C^g$:
$$\bar{ll}
\theta[\eta](w_1,\dots, w_g,\Omega_t) = & \theta[\eta_1](w_1,\dots,w_{g-1},\Omega)\\
& + \sqrt{t}\, e^{cw_g+ r}\,\theta[\eta_1](w_1+c_1,\dots,w_{g-1}+c_{g-1},\Omega) \\
& + \sqrt{t}\,e^{-cw_g - r}\,\theta[\eta_1](w_1-c_1,\dots,w_{g-1}-c_{g-1},\Omega) + O(t),
\ear$$
where $c,r,c_j$ are moduli-dependent constants and $c\cdot \theta[\eta_1](c_1,\dots,c_{g-1},\Omega)\neq 0$ outside of some divisor in the moduli space.
\end{prop}

\begin{rem}\label{anbound2}
Note that in the case $\delta = 0$ the family of bundles $L_t\to C_t,\ t\neq \mathbb R_{\geq 0},$ cannot be extended to a neighborhood of $t = 0$. To fix this one needs to take the double cover $s^2 = t$. Then, if $\eta$ is an odd characteristic, the family of spin curves $(C_s,L_s\to C_s),\ s^2 = t,\ t\in \mathbb C\smallsetminus\{0\},$ is a well-defined and its closure in $\cSo$ is transversal to the boundary divisor~$B_0$.
\end{rem}

The two propositions above follow directly from the asymptotics of $\Omega_t$ (see~\cite[p. 53]{FAY}):
\beq\label{irredbper}
\Omega_t = \begin{pmatrix}\Omega & R^T \\ R & \log t + c\end{pmatrix} + O(t),
\eeq
where $R\in \mathbb C^{g-1}$ and $c\in \mathbb C$ are moduli-dependent constants.

\section{Farkas' formula for $\Z$.}\label{farforsec}

\subsection{Odd spinors.}\label{ospin}

Consider a point in $\tAdo$ represented by a triple $(C, \a, \omega)$ as above. Then $\sqrt{\omega}$ is a section of an odd spin bundle~$L$. Denote by $\Omega$ the matrix of $b$-periods for $C$ with respect to $\a$. Let $\theta[\eta](\cdot,\Omega):\mathbb C^g\to \mathbb C$ be the theta function with the odd characteristic $\eta$ given by $L$ and~$\a$. Introduce the differential
$$
\s_C(p) = d_x\, \theta[\eta](\A (x - p),\Omega)|_{x = p}\,,
$$
where $\A$ is the Abel map (note that $\s_C(p)$ does not depend on a lift of $\A(x - p)$ to $\mathbb C^g$ since $\theta[\eta](0,\Omega) = 0$). This differential is non-zero if and only if $\dim H^0(C,L) = 1$ and is the square of a section of $L$. Therefore,
$$
\s_C = c\,\omega
$$
for some (moduli-dependent) constant $c$.

Let us describe the asymptotics of $\s$ under a degeneration of a curve.

{\bf The case of reducible curve}. Let $C = C_1\sqcup C_2/_{p_1\sim p_2}$ be an one-nodal reducible curve of genus $g$; we assume that $C$ represents a generic point in $\cM\smallsetminus \M$. Consider a curve $C_t$ constructed as in Section~\ref{thetasection} (see~\eqref{degenred}) and let $\a_i$ be a Torelli marking on $C_i$. Recall that $\a_1\cup \a_2$ induces a Torelli marking on $C_t$ for every $t\neq 0$. Fix an odd theta characteristic $\eta = \eta_1\oplus \eta_2$ such that $\eta_1\in (\mathbb Z/2\mathbb Z)^{2j}$ and $\eta_2\in (\mathbb Z/2\mathbb Z)^{2(g-j)}$, where $j$ is the genus of $C_1$ and $\eta_1$ is odd. Note that the condition that $\eta$ is odd implies that $\eta_2$ is even.

Let $K_i\subset C_i\smallsetminus\{p_i\}$ be a compact subset. We may assume that $K_i\subset C_t$ for all sufficiently small $t$. Then Proposition~\ref{thasred} implies that
\beq\label{asspinred1}
\s_{C_t}(p) = v_1(p) + O(t)
\eeq
uniformly on $K_1$ as $t\to 0$, where $v_1$ is a non-zero holomorphic differential on $C_1$, and
\beq\label{asspinred2}
\s_{C_t}(p) = t\,v_2(p) + O(t^2)
\eeq
uniformly on $K_2$ as $t\to 0$, where $v_2$ is a non-zero meromorphic differential on $C_2$ having double pole at $p_2$ and no other poles.

{\bf The case of irreducible curves.} Let $C$ be a smooth curve of genus $g-1$ and $p_1,p_2\in C$ be distinct points. Assume that $C/_{p_1\sim p_2}$ represents a generic point in $\cM\smallsetminus \M$. Consider the curve $C_t$ constructed as in Section~\ref{thetasection} (see~\eqref{degenirred}).

Fix an odd theta characteristic $\eta = \eta_1\oplus \begin{pmatrix} \eps \\ \delta \end{pmatrix}$ such that $\eta_1\in (\mathbb Z/2\mathbb Z)^{2(g-1)}$ and choose a Torelli marking on $C_t$ for small $t\in \mathbb D\smallsetminus \mathbb R_{\geq 0}$ as it was done in the previous subsection.

Let $K$ be a compact subset of $C\smallsetminus\{p_1,p_2\}$. Consider the following two cases:

{\it Case 1.} Let $\delta = 1$. In this case $\eta_1$ and $\eps$ have different parities. Let $L_t\to C_t$ be the spin bundle with characteristic $\eta$; then by Remark~\ref{anbound1} the pair $(C_t,L_t)$ is a family in $\cSo$ whose closure is transversal to $A_0$. Proposition~\ref{thasirred1} implies that $\s_{C_t}$ is determined up to an 8th root of unity and has the following asymptotics:
\beq\label{asspinirred1}
\s_{C_t}(p) = t^{1/8}(v(p) + O(t)),
\eeq
uniformly on $K$ as $t\to 0$, where $v$ is a non-zero meromorphic differential on $C$ having simple poles at $p_1$ and $p_2$ and no other poles.

{\it Case 2.} Let $\delta = 0$. Then $\eta_1$ must be odd. Introduce a new parameter $r = \sqrt{t}$. Let $L_r\to C_{t(r)}$ be the spin bundle with characteristic $\eta$; then by Remark~\ref{anbound2} the pair $(C_{t(r)},L_r)$ is a family in $\cSo$ whose closure is transversal to the boundary divisor $B_0$. Proposition~\ref{thasirred2} implies that $\s_{C_{t(r)}}$ is well-defined for all $r\neq 0$ and has the asymptotics
\beq\label{asspinirred2}
\s_{C_{t(r)}}(p) = v(p) + rv_1(p) + O(r^2)
\eeq
uniformly on $K$ as $r\to 0$, where $v$ is some holomorphic differential and $v_1$ is a meromorphic differential on $C$ having simple poles at $p_1,p_2$ and no other poles.

\bigskip

Let us analyze the global behavior of $\s$. Let $\nu: \tAdo\to \Ado$ be the forgetful projection. We first consider $\s$ as a section of the tautological line bubdle $\nu^*\L\to \tAdo/\mathbb C^*$.

Recall that the group $Sp(g,\mathbb Z)$ acts on $\tAdo$ by changing a Torelli marking, and we have $\tAdo/Sp(g,\mathbb Z) = \Ado$. The differential $\s$ transforms under the action of $Sp(g,\mathbb Z)$ in the following way:

\begin{prop}\label{oddspinor}
Let $(C, \a,L)$ be a Torelli marked curve, and $\sigma$ be a $Sp(g,\mathbb Z)$ - transformation acting on $H_1(C)$. Denote by $\begin{pmatrix} \sigma_{11} & \sigma_{12} \\ \sigma_{21} & \sigma_{22} \end{pmatrix}$ the matrix of $\sigma$ with respect to the basis $\a$. Then
$$
\s_{\sigma_* C} = \gamma\sqrt{\mathrm{det} (\sigma_{12}\Omega + \sigma_{11})}\cdot\s_C,
$$
where $\gamma^8 = 1$.
\end{prop}
\noindent The proposition follows directly from the transformation properties of theta functions (see~\cite{Mum}).

\begin{cor}\label{spinsec}
$\s^8$ can be considered as a section of the line bundle $\L^8\otimes\Lambda^4\to \Ado/\mathbb C^*$.
\end{cor}

\noindent We finalize with the following remark:

\begin{rem}
Let $\mu: \CMcal{C}^-_g\to \cSo$ be the universal spinor curve and $\omega^s$ be the line bundle on $\CMcal{C}^-_g$ such that $\omega^s$ is the corresponding spin bundle restricted to each fiber of $\mu$. Then $\mu_*\omega^s$ turns out to be a locally-free sheaf of the dimension one. $\s^8$ induces a section of the line bundle $(\mu_*\omega^s)^{16}\otimes\lambda^4$ restricted to $\So$. The asymptotics relations~\eqref{asspinirred1} -- \eqref{asspinred2} imply that this section can be extended to a section of $(\mu_*\omega^s)^{16}\otimes\lambda^4$ and the divisor of this section is $A_0$. But $\s^8$ considered as a section of $Sym^8\,\mathbb E_g^s\otimes\lambda^4$ (where $\mathbb E^s_g$ is the Hodge bundle on $\cSo$) has a bigger zero locus: it consists of $A_0$ and of the closure of the locus of spin curves $(C,L)\in \So$ such that $\dim H^0(C,L) >1$. This is connected with the fact that the pushforward functor is not right exact.
\end{rem}

\subsection{Asymptotics of the tau function.}\label{thas}

We begin with the following technical observation. Let $C$ be a Riemann surface of genus $g$ and $v$ be a holomorphic differential or a meromorphic differential with double poles and zero residues. Denote zeros of $v$ by $p_1,\dots,p_d\in C$. Consider simple paths $l_j$ from $p_d$ to $p_j$ for all $j = 1,\dots, d-1$. Let $a_1,\dots,a_g,b_1,\dots,b_g$ be simple loops on $C\smallsetminus\{p_1,\dots,p_d\}$ which do not intersect $l_j$ and such that their homology classes in $H_1(C)$ form a symplectic basis. Denote by $s_1,\dots,s_{2g+d-1}$ a basis in $H_1(C\smallsetminus\{p_1,\dots,p_d\})$ dual to the basis represented by $a_1,\dots,a_g,b_1,\dots,b_g,l_1,\dots,l_{d-1}$ in the relative homology group $H_1(C;\{p_1,\dots,p_d\})$; we have $s_j = -b_j,s_{g+j} = a_j$ and $s_{2g+j}$ is homologous to a small positive oriented circle around $p_j$.

Put
$$
\bar{l}
z_j = \int_{a_j} v,\quad j = 1,\dots,g,\\
z_{j+g} = \int_{b_j} v,\quad j = 1,\dots,g,\\
z_{j+2g} = \int_{l_j} v,\quad j = 1,\dots,d-1.
\ear
$$
In the case when $v$ is a holomorphic differential with double zeros the set $\{z_1,\dots,z_{2g+d-1}\}$ is the set of homological coordinates introduced above.

Let $S_B$ be the Bergman projective connection with respect to the Torelli marking induced by $a_1,b_1,\dots,a_g,b_g$. Denote by $m_k$ the multiplicity of the zero $p_k$.
\begin{lemma}\label{rbr}

The following relation holds:
\beq\label{homogengen}
\sum\limits_{k = 1}^{2g+d-1} z_k \int_{s_k} \frac{S_B - S_v}v = -\pi i \left ( d +  \sum\limits_{k = 1}^d (m_k- \frac{1}{1+m_k}) \right )
\eeq
\end{lemma}

\begin{proof} From Riemann bilinear relations we get that
$$
\sum\limits_{k = 1}^{2g} z_k \int_{s_k} \frac{S_B - S_v}v = -2\pi i\sum\limits_{x\in C}\mathrm{Res}_x\, \left (\frac{S_B - S_v}v\int_{p_d} v\right ).
$$
Computing residues we obtain
$$
-2\pi i\sum\limits_{x\in C}\mathrm{Res}_x\, \left (\frac{S_B - S_v}v\int_{p_d} v\right ) = -\sum\limits_{k = 2g + 1}^{2g + d-1} z_k \int_{s_k} \frac{S_B - S_v}v - \pi i \left ( d +  \sum\limits_{k = 1}^d (m_k- \frac{1}{1+m_k}) \right )
$$
which implies~\eqref{homogengen}.
\end{proof}
\begin{rem}
If $v$ is holomorphic differential with double zeros then the right-hand side of~\eqref{homogengen} is equal to $\frac 83 (1-g)$. This implies the homogeneity property of the tau function.

In fact~\eqref{homogengen} implies that if a function $F$ is defined on some open subset $\U\subset \tAdo$ and satisfies differential equations
$$
\partial_{z_j}\,\log F(C,v) = \frac{-\alpha}{\pi i} \int_{s_j} \frac{S_B - S_v}v,\quad j = 1,\dots, 2g+d-1,
$$
for some $\alpha\in \mathbb Q$, then it must satisfy the homogeneity property
$$
F(C,tv) = t^{\alpha\left ( d +  \sum\limits_{k = 1}^d (m_k- \frac{1}{1+m_k}) \right )} F(C,v).
$$
\end{rem}

\begin{prop}
\label{as1}
Consider a family of Torelli marked curves $(C_t,\a_t)$ in $\Mt$ and an odd theta characteristic $\eta\in (\mathbb Z/2\mathbb Z)^{2g}$ such that $C_0$ with the spin bundle $L_0\to C_0$ represents a point in $\Z$ and $t\in \mathbb C$ is transversal to $\Z$. Put $\s_{C_t} = \s_t$ for simplicity. Assume that $\s_0\neq 0$ (i. e. $\dim H^0(C_0,L_0) = 1$). Then the tau function $\tau$ has the following asymptotics near $\Z$:
\beq
\tau(C_t,\s_t) = c_0\,t^8(1+ o(1))\quad \text{as}\ t\to 0.
\eeq
\end{prop}
\begin{proof}
We may assume that $C_t$ defines a family of complex structures on a fixed topological surface. Let $p_{g-2}(t),p_{g-1}(t)\in C$ be the zeros of $\s_t$ that coalesce when $t\to 0$. Introduce a local coordinate $z_t:U\to\mathbb C$ on $C_t$ near $p_{2g-2}(0)$ such that $z_t(p_{g-2}(t)) = \sqrt{t}$ and such that $z_t(p_{g-1}(t)) = -\sqrt{t}$. Note that the point $(C_t,\s_t)$ in $\Ado$ does not depend on a labeling of zeros, so in our case we have a double cover on which $\sqrt{t}$ make sense. Then one has $\s_t\circ z_t^{-1} (x) = (x^2-t)(x^2+t)(c + O(t))\,dx$ for some $c\neq 0$ and therefore
$$
\int _{p_{g-2}(t)}^{p_{g-1}(t)}\s_t = t^{5/2}(c_1+O(t)),
$$
where the path of integration is chosen such that $\int _{p_{g-2}(t)}^{p_{g-1}(t)}\s_t \to 0$.

Let $z_1(t),\dots,z_{3g-2}(t)$ be the homological coordinates associated with the triple $(C_t,\a_t,\s_t)$ for $t\neq 0$. We may assume that $z_{3g-2}(t) = t^{5/2}(c_1 + O(t))$. Consider a small open neighborhood $\U\subset \So$ of $(C_0,L_0)$. Then calculations above imply that the map
$$
\xymatrix{
\U\ar^(0.4){[z_1:\dots:z_{3g-3}:z_{3g-2}^{2/5}]}[rrr] &&& \mathbb CP^{3g-3}
}
$$
is an embedding and the image of $\Z\cap \U$ is given by the intersection with the hyperplane $\{z_{3g-2} = 0\}$.

Denote the image of $\U$ in $\mathbb CP^{3g-3}$ by $\V$ and the pullback of $\V$ to $\mathbb C^{3g-2}$ by $\tilde{\V}$. The function $\tau$ written in local coordinates $z_1,\dots,z_{3g-2}$ can be considered as a function on the two-sheeted cover of $\tilde{\V}\smallsetminus \{z_{3g-2} = 0\}$ which is defined by the square root $\sqrt{z_{3g-2}}$. The relation~\eqref{taudefin} implies that
\beq\label{factzg}
\tau(z_1,\dots,z_{3g-2}) = c(z_{3g-2}^{2/5})\tilde{\tau}(z_1,\dots,z_{3g-3})\,(1+o(1))
\eeq
as $z_{3g-2}\to 0$ where $c$ is a meromorphic function having a singularity at the origin and $\tilde{\tau}(z_1,\dots,z_{3g-3})$ is a holomorphic function ($\tilde{\tau}$ is nothing but 72th power of the Bergman tau function considered on the stratum of holomorphic differentials on genus $g$ surfaces having $g-3$ double zero and one zero of order $4$. This stratum projects to a dense open subset of $\Z$). A simple estimate shows that $z\frac{d}{dz}\log c(z)$ is bounded and therefore $c$ must be meromoprhic near the origin.

Lemma~\ref{rbr} applied to the function $\tilde{\tau}$ implies that $\tilde{\tau}$ is homogenous with the degree of homogeneity equal to $16(g-1) - \frac{16}5$. Therefore comparing the degree of homogeneity of the left-hand side and the right-hand side of~\eqref{factzg} one concludes that $c(z) = z^8\,(c_0+o(1))$.
\end{proof}

Fix $0 < j < g$. Let $C = C_1\sqcup C_2/_{x_1\sim x_2}$ be a one-nodal curve of genus $g$ where $C_1,C_2$ are smooth and the genus of $C_1$ is $j$. Consider a curve $C_t$ constructed as in~\eqref{degenred}. Let us fix a Torelli marking $\a_1\cup \a_2$ on $C_t$ as above and there an odd theta-characteristic $\eta\in (\mathbb Z/2\mathbb Z)^{2g}$ such that the projection of $\eta$ to the first $2j$ components is odd cahracteristic. From this data we define $\s_{C_t}$ which we briefly denote by $\s_t$. Assume that $\s_0\neq 0$.

\begin{prop}\label{aj}
The tau function $\tau$ has the following asymptotics near $A_j\cup B_j$, $j>0$:
\beq\label{ajbjcase}
\tau(C_t, \s_t) = c\,t^{16(g -j)}(1+ o(1))\quad \text{as}\ t\to 0.
\eeq
\end{prop}
\begin{proof}
Recall that on any compact subset of $C_2\smallsetminus\{x_2\}$ one has
$$
t^{-1}\,\s_t \to v_2
$$
as $t\to 0,$ where $v_2$ is a meromorphic differential on $C_2$ with a double pole at $x_2$ and no other poles (see~\eqref{asspinred2}). Fix some enumeration $p_1(t),\dots,p_{g-1}(t)$ of zeros of $\s_t$ such that $p_1(t),\dots,p_{g-j-1}(t),p_{g-1}(t)\in C_2$. Let $z_1,\dots,z_{3g-2}$ be homological coordinates constructed with respect to $\a_1\cup\a_2$ and the chosen numeration of zeros. Then direct computations using the differential equation~\eqref{taudefin} and asymptotics relations~\eqref{asspinred1} and~\eqref{asspinred2} give
$$
\frac d{dt} \log \tau(C_t, \s_t) = - t^{-1}\cdot \frac{6}{\pi i} \sum\limits_{k = 1}^{d} z_k \int_{s_k} \frac{S_B - S_{v_2}}{v_2} + O(1)\quad \text{as}\ t\to 0,
$$
where $S_B$ is the Bergman projective connection, $d = 3g-j-1$ and $s_1,\dots,s_d$ is the basis in $H_1(C_2\smallsetminus\{p_1(0),\dots,p_{g-j-1}(0),p_{g-1}(0)\})$ dual to the basis in the relatives homology group defining homological coordinates. By Lemma~\ref{rbr}  one sees that
$$
\frac d{dt} \log \tau(C_t, \s_t) = \frac{16(g-j)}t + O(1),
$$
which implies~\eqref{ajbjcase}.

\end{proof}

We now consider the irreducible case. Let $C$ be a smooth genus $g-1$ curve and $p_1,p_2\in C$ be distinct points. Then $C/_{p_1\sim p_2}$ is a one-nodal irreducible curve of genus $g$. Consider a family of curves $C_t$ constructed as above (see~\eqref{degenirred}). Fix a Torelli marking $\a$ on $C$. Choose as above a simple path $b$ from $p_1$ to $p_2$ and a simple loop $a$ around $p_1$. Then $\{a,b\}\cup\a$ induces a Torelli marking on $C_t$ for all $t\in \mathbb D\smallsetminus \mathbb R_{\geq 0}$.

Fix an odd theta-characteristic $\eta = \eta_1\oplus \begin{pmatrix}\eps \\ 1\end{pmatrix}$. Then $\eta$ and the chosen Torelli marking determines $\s_{C_t}$; $\sqrt{\s_{C_t}}$ is a section of a spin bundle $L_t\to C_t$ such that $(C_t,L_t)$ tends to some point in the boundary divisor $A_0$ as $t\to 0$. Fix some branch of $t^{1/8}$. Recall that $\frac{1}{t^{1/8}}\s_{C_t}\to v$ as $t\to 0$ for some (generically not identically vanishing) meromorphic differential $v$ on $C$ having simple poles at $p_1,p_2$ and no other poles (see~\eqref{asspinirred1}). We denote $\frac{1}{t^{1/8}}\s_{C_t}$ by $\tilde{\s}_t$.

\begin{prop}
The tau function $\tau$ has the following asymptotics near $A_0$:
\beq\label{aocase}
\tau(C_t,\tilde{\s}_t) = c\,t^6(1 + o(1))\quad \text{as}\ t\to 0.
\eeq

\end{prop}

\begin{proof}
Let $z_1(t) = \int_a\tilde{\s}_t$ and $z_2(t) = \int_b\tilde{\s}_t$. Consider the parameter
$$
\tilde{t} = \exp\left ( 2\pi i\frac{z_2(t)}{z_1(t)} \right ).
$$
Recall that when $t$ goes around zero then $b$ changes to $b+a$ and $\tilde{\s}_t$ to $\gamma\tilde{\s}_t$ where $\gamma^8 = 1$. This implies that $\tilde{t}$ can be naturally extended as a function of $t$ for all $t\in\mathbb D$. The asymptotics $\int_b\tilde{\s}_t = \frac{z_1(0)}{2\pi i}\log t + O(1)$ (see~\eqref{irredbper}) implies that $\tilde{t}(0) = 0$ and $\tilde{t}(t)$ is one-to-one map near the origin.

We fix some labeling of zeros of $\tilde{\s}_t$ and introduce the corresponding homological coordinates. Note that $\tau(C_t,\tilde{\s}_t)$ is correctly defined for all sufficiently small $t\in \mathbb C$. Using the equation~\eqref{taudefin} defining the tau function we compute by the chain rule that
$$
-\frac{\pi i} 6\cdot \frac d{d \tilde{t}}\log \tau(C_{\tilde{t}}, \tilde{\s}_{\tilde{t}}) = \frac{z_1(0)}{2\pi i\,\tilde{t}} \int_a\frac{S_B - S_{\tilde{\s}_{\tilde{t}}} }{\tilde{\s}_{\tilde{t}}}\cdot(1 + o(1))
$$
as $t\to 0$. Computing the residue $\mathrm{Res}_{p_1}\,\frac{S_B - S_{\tilde{\s}_0 }}{\tilde{\s}_0}$ we obtain
$$
\frac d{d \tilde{t}}\log \tau(C_{\tilde{t}}, \tilde{\s}_{\tilde{t}}) = \frac6{\tilde{t}}(1+o(1)),
$$
which implies~\eqref{aocase}.
\end{proof}

Consider now an odd theta characteristic $\eta = \eta_1\oplus\begin{pmatrix}\eps \\ 0\end{pmatrix}$. Then $\eta$ and the chosen Torelli marking determine $\s_{C_t}$; $\sqrt{\s_{C_t}}$ is a section of a spin bundle $L_t\to C_t$ such that $(C_t,L_t)$ tends to some point $(C_0,L_0)$ in $B_0$ as $t\to 0$. Consider a new parameter $r = \sqrt{t}$. Recall that by~\eqref{asspinirred2} there exists a holomorphic differential $v$ on $C$ and a meromorphic differential $w$ on $C$ having simple poles at $p_1$ and $p_2$ and no other poles such that $\s_{C_{t(r)}} = v + rw + O(r^2)$. Put $\s_r = \s_{t(r)}$ and $C_r = C_{t(r)}$ to simplify notations.

\begin{prop}
\label{as5}
The tau function $\tau$ has the following asymptotics near $B_0$:
$$
\tau(C_r,\s_r) = c\,r^{16}(1 + o(1))\qquad \text{as}\ r\to 0.
$$
\end{prop}

\begin{proof}
Let $\U$ be a small open polydisc in $\cSo$ centered at $(C_0,L_0)$ and let $\V$ be a connected component of the pullback of $\U$ to $\tAdo$. Introduce homological coordinates $z_1,\dots,z_{3g-2}$ on $\V$ that are numbered as follows:
$$
z_g\left(C_r, \{a,b\}\cup\a, \s_{C_r}\right) = \int_a \s_r,\quad z_{2g}\left(C_r, \{a,b\}\cup\a, \s_r\right) = \int_b \s_r
$$
and the $(g-1)$th zero of the differential $\s_r$ tends to the node under a degeneration of the underlying curve. Note that by the asymptotics~\eqref{asspinirred2}
$$
z_g\left(C_r, \{a,b\}\cup\a, \s_r\right) = c\,r(1+o(1))
$$
for some generically non-zero constant $c$. The asymptotics~\eqref{asspinirred2} implies that
\beq\label{poleandnomore}
r\,\int_b\frac{S_B - S_{\s_r} }{\s_r} = O(1)
\eeq
as $r\to 0$.

Computing the derivatives of $\tau$ with respect to $z_j$ for all $j\neq g,2g$ by~\eqref{taudefin}, we obtain the asymptotics
\beq\label{facr}
\tau(z_1,\dots,z_{3g-2}) = c(z_g,z_{2g})\,\tilde{\tau}(z_1,\dots,\hat{z}_g,\dots,\hat{z}_{2g},\dots,z_{3g-3})(1+o(1))\quad\text{as }z_g\to 0,
\eeq
where $\tilde{\tau}$ is the tau function on $\tilde{\CMcal{H}}_{g-1}^{-}([2^{g-2}])$.

The factor $c(z_g,z_{2g})$ is a holomorphic function in some punctured neighborhood of the line $\{(0,z),~z~\in~\mathbb C\}$ in $\mathbb C^2$. The estimate~\eqref{poleandnomore} shows that $\frac{\partial}{\partial z_{g}}\log \tau$ has at most simple pole at $z_g = 0$, hence the function $c(z_g,z_{2g})$ is meromorphic at $z_g=0$. Consider the Laurent series
$$
c(z_g,z_{2g}) = \sum\limits_{j = N}^{+\infty} c_j(z_{2g})z_g^j,
$$
It follows from the differential equation defining $\tau$ that $\frac{\partial}{\partial z_{2g}}\log\tau = O(z_g)$; therefore $c_{N}$ does not depend on $z_{2g}$. According to Lemma~\ref{tauprop} the degree of homogeneity of $\tau$ under the $\mathbb C^*$-action on differentials is equal to $16(g-1)$, whereas the degree of homogeneity of $\tilde{\tau}$ is equal to $16(g-2)$. Thus, comparing the orders of homogeneity of the right-hand side and the left-hand side of~\eqref{facr} one sees that $N = 16$.

\end{proof}

\subsection{The formula.}

Now we can prove the following statement originally obtained by G. Farkas~\cite{FARo}:

\begin{thma}
The class $[\Z]$ has the following expression via the standard basis of the rational Picard group of $\cSo$:
\beq\label{farforlab}
[\Z] = (g+8)\lambda - \frac{g+2}{4}\alpha_0 - 2\beta_0 - \sum\limits_{j = 1}^{\left [ g/2\right ]}2(g-j)\alpha_j - \sum\limits_{j = 1}^{\left [ g/2\right ]}2j\beta_j.
\eeq
\end{thma}
\begin{proof}
Note that $\s^{16(g-1)}$ is a section of the line bundle 
$$\xymatrix{\L^{16(g-1)}\otimes \Lambda^{8(g-1)} \ar[d]\\
\Ado/\mathbb C^*}
$$
as it was shown in Corollary~\ref{spinsec} (see Subsection~\ref{ospin} for the definition of~$\s$). By Corollary~\ref{secdef} the tau function defines a homomorphism from $\L^{16(g-1)}$ to $\Lambda^{72}$. Applying this homomorphism to the section $\s^{16(g-1)}$ we obtain a section of $\Lambda^{8g+ 64}$ which we denote by $\tilde{\psi}$.

Consider the locus $\CMcal{X} = \{(C,\omega)\in \Ado\ :\ \dim |\div\,\sqrt{\omega}| > 0\}$ (that is, the locus of abelian differentials with double zeros such that the dimension of the space of holomorphic sections of the corresponding spin bundle is larger than one). Note that $\pi|_{\Ado\smallsetminus \CMcal{X}}$ is one-to-one and $\pi(\Ado\smallsetminus \CMcal{X}) = \So\smallsetminus \Z$, where $\pi$ is the map from $\Ado$ to $\So$ which maps a differential to the corresponding spin bundle. We also have $\pi(\CMcal{X})\subset \Z$.

Put $\psi = \pi_*(\tilde{\psi}|_{\Ado\smallsetminus \CMcal{X}})$. We have $\pi_*\Lambda^{8g+64} \simeq \lambda^{8g+64}$, therefore $\psi$ is a holomorphic section of $\lambda^{8g+64}|_{\So\smallsetminus \Z}$. Let $\U\subset \cSo$ be an open contractible subset. Choosing a trivialization $\phi:\lambda|_{\U}\to \U\times \mathbb C$ we obtain a holomorphic function $\phi^{\otimes 8g+64}\circ \psi:\U\cap  (\So\smallsetminus \Z)\to \mathbb C$. Propositions~\ref{as1} -- \ref{as5} and asymptotics~\eqref{asspinirred1} -- \eqref{asspinred2} imply that this function can be holomorphicaly extended to $\U$. Therefore we can extend the section $\psi$ to $\cSo$. Propositions~\ref{as1} -- \ref{as5} and asymptotics~\eqref{asspinirred1} -- \eqref{asspinred2} also imply that
$$
[\div\,\psi] = 16\beta_0 + (4+ 2g)\alpha_0 + 16 \sum\limits_{j = 2}^{\left[ g/2 \right]} (g-j)\alpha_j + 16 \sum\limits_{j = 2}^{\left[ g/2 \right]} j \beta_j + 8[\Z].
$$
On the other hand,
$$
[\div\, \psi] = (8g+ 64)\lambda
$$
in the rational Picard group of $\cSo$ by definition of $\psi$. Hence
$$
(8g+ 64)\lambda = 16\beta_0 + (4+ 2g)\alpha_0 + 16 \sum\limits_{j = 2}^{\left[ g/2 \right]} (g-j)\alpha_j + 16 \sum\limits_{j = 2}^{\left[ g/2 \right]} j \beta_j + 8[\Z].
$$
which implies Formula~\eqref{farforlab}.

\end{proof}

\section{A formula for the theta-null divisor.}\label{thetanullsection}

Here we give an alternative proof of Farkas' formula for the class of the theta-null divisor (see~\cite{FARe}) using the modular properties  of the theta function. Let $\cSe$ be the moduli space of even spin curve of genus $g$ and let $\Se\subset \cSe$ be the subspace of smooth spin curves. Consider the theta-null divisor:
$$
\Thn = Cl\{(C,L)\in \Se\ :\ \dim\, H^0(C,L) > 0\},
$$
where the closure is taken in $\cSe$.

{\bf The rational Picard group of $\cSe$.} Let $\chi : \cSe\to \cM$ be the natural projection. The boundary $\cSe\smallsetminus \Se$ is a union of irreducible divisors $A^+_0,B^+_0,\dots,A^+_{\left[g/2\right]},B^+_{\left[g/2\right]}$ such that $\chi (A^+_j) = \chi (B^+_j) = \Delta_j$ for all $j = 0,\dots,\left [ \frac g2 \right ]$.

If $j\neq 0$ then a generic point in $A^+_j$ is represented by an even spin bundle on each of the two irreducible components of a reducible genus $g$ curve with one node. Generic points in $B^+_j$ are similarly represented by odd spin bundles. In these cases we also replace the node by an exceptional component. Formally,
$$
\bar{l}
A^+_j = Cl\{(C_1\cup E\cup C_2, \eta, \beta)\in \cSe\ :\ \eta|_{C_1} \text{ is even} \},\\
B^+_j = Cl\{(C_1\cup E\cup C_2, \eta, \beta)\in \cSe\ :\ \eta|_{C_1} \text{ is odd} \},
\ear
$$
where $C_1,C_2$ are smooth, the genus of $C_1$ is $j$, and $E$ is rational.

The divisor $A^+_0$ is the closure of the locus of even spin curves with one node (the underlying curve must be irreducible in this case). Formally,
$$
\bar{l}
A^+_0 = Cl\{  (C/_{p\sim q}, \eta, \beta)\in \cSe\ :\ \text{$C$ is smooth of genus $g-1$},\ p,q\in C   \},\\
B^+_0 = Cl\{  (C\cup E, \eta, \beta)\in \cSe\ :\ \text{$C$ is smooth of genus $g-1$, $E$ is exceptional}  \}.\\
\ear
$$

Let $\mathbb E_g$ be the pullback of the Hodge vector bundle from $\cM$ to $\cSe$ and let $\lambda$ be the class in $\Pic(\cSe)\otimes\mathbb Q$ of the determinant bundle $\bigwedge^g\mathbb E_g$. Denote by $\alpha^+_j$ and $\beta_j^+$ the classes of $A^+_j$ and $B_j^+$ in the rational Picard group respectively. The group $\Pic(\cSe)\otimes\mathbb Q$ is generated by $\lambda,\alpha_0^+,\dots,\alpha^+_{\left[g/2\right]},\beta_0^+,\dots,\beta^+_{\left[g/2\right]}$.

{\bf Theta function as a modular form.} Consider the cover $\Ste$ of $\Se$ induced by the forgetful map $\Mt\to\M$. A point in $\Ste$ is represented by a triple $(C,\a,\eta)$, where $C$ is a genus $g$ smooth curve, $\a$ is a Torelli marking and $\eta\in (\mathbb Z/2\mathbb Z)^{2g}$ is an even theta characteristic (that is $\sum\limits_{j = 1}^g \eta_{2j}\eta_{2j-1} = 0$). Thus each point $(C,\a,\eta)$ of $\Ste$ defines a theta function with characteristics:
$$
\theta[\eta](\cdot,\Omega):\mathbb C^g\to \mathbb C,
$$
where $\Omega$ is the matrix of $b$-periods with respect to $\a$. Introduce a notation $\t(C,\a,\eta) = \theta[\eta](0,\Omega)$. Then $\t$ is a holomorphic function on the space $\Ste$.

There is a natural action of the symplectic group $Sp(g,\mathbb Z)$ on $\Ste$ that changes a Torelli marking. Note that $\Ste/Sp(g,\mathbb Z) = \Se$.

\begin{prop}\label{thetamod}
Fix a point $(C,\a,\eta)$ in $\Ste$. Let $\sigma$ be a $Sp(g,\mathbb Z)$ - transformation of $H_1(C)$ and $\begin{pmatrix} \sigma_{11} & \sigma_{12} \\ \sigma_{21} & \sigma_{22} \end{pmatrix}$ be its matrix with respect to the basis $\a$. Then
$$
\t(C,\sigma(\a),\sigma_*(\eta)) = \gamma\sqrt{\mathrm{det} (\sigma_{12}\Omega + \sigma_{11})}\cdot\t(C,\a,\eta),
$$
where $\gamma^8 = 1$.
\end{prop}
\noindent The proposition immediately follows from transformation properties of theta functions (see~\cite{Mum}).

The above proposition implies that $\t^8$ is a holomorphic section of $(\bigwedge^g\mathbb E_g)^{\otimes 4}|_{\Se}$. From the classical Riemann theorem we get that the divisor of this section is equal to $n\cdot (\Thn\cap \Se)$ for some $n\in \mathbb Z_{>0}$. It is well-known that $n = 16$ (see~\cite{Tei}).

The computations of Section~\ref{thetasection} imply that $\t^8$ extends to $\cSe$ as a holomorphic section of $(\bigwedge^g\mathbb E^g)^{\otimes 4}$ and
\beq\label{thdiv}
[\div_{\cSe}\,\t^8] = 16[\Thn] + \alpha^+_0 + 8\sum\limits_{j = 1}^{\left [ g/2 \right ]} \beta_j^+,
\eeq
where $[\ ]$ denotes the class of a divisor in the rational Picard group.

\begin{thma}
We have
\beq\label{mainformulafortheta}
[\Thn] = \frac 14\lambda - \frac 1{16}\alpha^+_0 - \frac12\sum\limits_{j = 1}^{\left [ g/2 \right ]} \beta_j^+.
\eeq
\end{thma}
\begin{proof}
Since $\t^8$ is a section of $(\bigwedge^g\mathbb E^g)^{\otimes 4}$,
$$
[\div_{\cSe}]\,\t^8 = 4\lambda,
$$
from where~\eqref{mainformulafortheta} immediately follows.
\end{proof}

{\bf Acknowledgements.} The author is grateful to D. Korotkin for proposing the problem and general guidance. The author also thanks A. Kuznetsov for a helpful review in algebraic geometry and P. Zograf for many useful suggestions.

\end{document}